\documentclass[12pt]{amsart}
\usepackage{amscd,amsmath,amssymb,amsthm,amsfonts,cite,epsfig,graphics}
\usepackage{graphicx}
\usepackage{enumerate}
\usepackage{epstopdf}
\usepackage{mathtools}
\usepackage{caption}
\usepackage{subcaption}
\usepackage{xcolor}

\usepackage{import}
\usepackage{xifthen}
\usepackage{pdfpages}
\usepackage{transparent}
\usepackage{verbatim}

\usepackage{amssymb}
\usepackage{url}
\usepackage{amscd}
\usepackage{texdraw}
\usepackage{tikz-cd}
\usepackage{tikz}

\theoremstyle{plain}

\newtheorem{theorem}{Theorem}
\newtheorem{corollary}{Corollary}
\newtheorem{lemma}{Lemma}
\newtheorem{remark}{Remark}

\newtheorem{definition}{Definition}
\newtheorem{example}{Example}
\newtheorem{conjecture}{Conjecture}

\def\o{\omega}

\def\S{\Sigma}
\def\s{\sigma}
\def\CE{\mathcal{CE}}
\def\CH{\mathcal{CH}}
\def\QS{\mathcal{QS}}
\def\S{\mathcal{S}}
\def\UQCE{\mathcal{UQCE}}
\def\USCE{\mathcal{USCE}}
\def\BGCE{\mathcal{BGCE}}
\def\CED{\mathcal{CED}}

\begin{document}
\title[Symmetric Rigidity]{Symmetric Rigidity for Circle Endomorphisms with Bounded Geometry}
\author[Adamski, Hu, Jiang, Wang]{John Adamski, Yunchun Hu, Yunping Jiang, and Zhe Wang}

\address{John Adamski: Department of Mathematics, Fordham University} 
\email{jadamski1@fordham.edu}

%\address{Fred Gardiner: Department of Mathematics, CUNY Graduate Center}
%\email{ffrederick.gardiner@gmail.com}

\address{Yunchun Hu: Department of Mathematics and Computer Science, CUNY Bronx Community College} 
\email{yunchun.hu@bcc.cuny.edu}

\address{Yunping Jiang: Department of Mathematics, CUNY Queens College and Graduate Center}
\email{yunping.jiang@qc.cuny.edu}

\address{Zhe Wang: Department of Mathematics and Computer Science, CUNY Bronx Community College}
\email{zhe.wang@bcc.cuny.edu}

\thanks{This material is based upon work supported by the National Science Foundation. It is also partially supported by a
collaboration grant from the Simons Foundation (grant number 523341) and PSC-CUNY awards.} 



\begin{abstract}
Let $f$ and $g$ be two circle endomorphisms of degree $d\geq 2$ such that each has bounded geometry, preserves the Lebesgue measure, and fixes $1$. 
Let $h$ fixing $1$ be the topological conjugacy from $f$ to $g$. That is, $h\circ f=g\circ h$. 
We prove that $h$ is a symmetric circle homeomorphism if and only if $h=Id$.
Many other rigidity results in circle dynamics follow from this very general symmetric rigidity result.  
\end{abstract}

\maketitle

\section{Introduction}

A remarkable result in geometry is the so-called Mostow rigidity theorem. This result assures that two closed hyperbolic $3$-manifolds are isometrically equivalent if they are homeomorphically equivalent~\cite{Mo}.
A closed hyperbolic $3$-manifold can be viewed as the quotient space of a Kleinian group acting on the open unit ball in the $3$-Euclidean space. So a homeomorphic equivalence between two closed hyperbolic $3$-manifolds
can be lifted to a homeomorphism of the open unit ball preserving group actions. The homeomorphism can be extended to the boundary of the open unit ball as a boundary map. The boundary is the Riemann sphere and the boundary map is a quasi-conformal homeomorphism. A quasi-conformal homeomorphism of
the Riemann sphere is absolutely continuous. It follows that the boundary map has no invariant line field, and thus it is a M\"obius transformation. 

For closed hyperbolic Riemann surfaces, the situation is quite complicated.
A closed hyperbolic Riemann surface can be viewed as a Fuchsian group acting
on the open unit disk in the complex plane. A homeomorphic equivalence between two closed
hyperbolic Riemann surfaces can be lifted to a homeomorphism of the open unit disk preserving group actions. This homeomorphism can be extended to the
boundary of the open unit disk as a boundary map. In this case, the boundary
is the unit circle and the boundary map is a quasisymmetric homeomorphism. The complication is due to the fact that a quasisymmetric homeomorphism may not be absolutely continuous. Complicated maps like this are a rich source for Teichm\"uller theory. However, if we assume that the boundary map is absolutely continuous, then by following the main idea in the proof of Mostow’s rigidity theorem, it is a M\"obius transformation.

Shub and Sullivan further developed the study of conjugacies between smooth expanding circle endomorphisms in~\cite{SS}. The conjugacy in this case is always quasisymmetric (refer to~\cite[Chapter 3]{JiangBook}, and see also~\cite{JiangGeo,JiangGibbs}). They proved that if the conjugacy is absolutely continuous then it is smooth. One of the authors (Jiang) studied the smoothness of conjugacies between one-dimensional maps with singularities. In~\cite{JiangUlam}, he first proved that the conjugacy between two generalized Ulam–von Neumann transformations is smooth if their power-law singularities have the same exponents, their asymmetries are the same, and their eigenvalues at all corresponding periodic points are the same. Later, the hypothesis that their asymmetries are the same was removed  in~\cite{JiangSm, JiangRig}. Moreover, in~\cite{JiangGeo, JiangSm, JiangRig}, he studied the smoothness of the conjugacy between two geometrically finite one-dimensional maps and proved that the conjugacy between two geometrically finite one-dimensional maps is always quasisymmetric. He also defined a smooth invariant called the scaling function for a geometrically finite one-dimensional map and proved that the scaling function and the exponents of power-law singularities are complete smooth invariants. That is, the conjugacy between two geometrically finite one-dimensional maps is smooth if and only if the maps have the same scaling function and the exponents of the corresponding power-law singularities are the same. More results on differential properties and the symmetric properties of a conjugacy between two one-dimensional maps are given in~\cite{JiangDiff,HuProc}.  Finally, the symmetric regularity of the conjugacy between two one-dimensional maps (with possible singularities) becomes an important issue in the study of one-dimensional dynamical systems, in particular in the study of geometric Gibbs theory (see~\cite{JiangSymm,JiangIntro,JiangGibbs}). We would like to note that a symmetric homeomorphism is, in general, not absolutely continuous.  The following conjecture is stated in~\cite[Conjecture 2.4]{JiangSymm} and~\cite[Conjecture10.12]{JiangGibbs}. 

\medskip
\begin{conjecture}~\label{conj1}
Suppose $f$ and $g$ are two uniformly symmetric circle endomorphisms of the same degree $d\geq 2$ such that $f(1)=g(1)=1$.
Suppose both $f$ and $g$ preserve the Lebesgue measure on the circle. 
Suppose $h$ is the conjugacy from $f$ to $g$ with $h(1)=1$. That is, $h\circ f=g\circ h$. 
Then $h$ is a symmetric circle homeomorphism if and only if $h$ is the identity.
\end{conjecture}

The paper~\cite[Theorem 2.5]{JiangSymm} (see also~\cite[Corollary 10.9]{JiangGibbs}) 
gives a partial proof of Conjecture~\ref{conj1}, and many discussions about symmetric rigidity 
in the smooth case are also given. In our study of this conjecture, we extended our research into uniformly quasisymmetric circle endomorphisms in~\cite{JiangNote} and posed a more general conjecture 
(see~\cite{JiangNote} and~\cite[Conjecture 2]{HJW}) as follows.

\medskip
\begin{conjecture}~\label{conj2}
Suppose $f$ and $g$ are two uniformly quasisymmetric circle endomorphisms of the same degree $d\geq 2$ such that $f(1)=g(1)=1$. 
Suppose both $f$ and $g$ preserve the Lebesgue measure on the circle. 
Suppose $h$ is the conjugacy from $f$ to $g$ with $h(1)=1$. That is, $h\circ f=g\circ h$.  
Then $h$ is a symmetric circle homeomorphism if and only if $h$ is the identity.
\end{conjecture}

%We borrowed the martingale theory into this study and proved many characterization results about uniformly quasisymmetric circle endomorphisms and their associated martingales in~\cite{HJW,HuThesis,AdamskiThesis}. 
%The paper~\cite{AHJW} gives a partial answer to Conjecture~\ref{conj2}. 

%Solving any one of these two conjectures becomes an important and interesting problem. We will give a complete answer to these two conjectures in this paper. Actually, we will prove a more general theorem as follows. 
In this paper, we will prove both of these conjectures completely by proving the following more general theorem.

\medskip
\begin{theorem}[Main Theorem]~\label{main}
Suppose $f$ and $g$ are two circle endomorphisms having bounded geometry of the same degree $d\geq 2$ such that $f(1)=g(1)=1$. 
Suppose $f$ and $g$ preserve the Lebesgue measure on the unit circle. 
Let $h$ be the conjugacy from $f$ to $g$ with $h(1)=1$. That is, $h\circ f=g\circ h$. 
If $h$ is a symmetric homeomorphism, then $h$ is the identity.
\end{theorem}

Since a uniformly symmetric circle endomorphism is uniformly quasisymmetric, and a uniformly quasisymmetric circle endomorphism has bounded geometry (see~\cite{JiangNote,JiangGibbs}), Theorem~\ref{main} gives an affirmative answer to both Conjecture~\ref{conj1} and Conjecture~\ref{conj2} (see Corollary~\ref{cor1} and Corollary~\ref{cor2}).  Our main theorem (Theorem~\ref{main}) also gives new proofs of the previous symmetric rigidity results (Corollary~\ref{cor3} and Corollary~\ref{cor4}) for the smooth case which were proved in~\cite{JiangSymm} by using transfer operators.  
  
We organize this paper as follows. In Section~\ref{cehbg}, we define a circle endomorphism having bounded geometry. In the same section, we review the definition of a uniformly quasisymmetric circle endomorphism and the definition of a uniformly symmetric circle endomorphism. We also review the definition of a $C^{1+Dini}$ expanding circle endomorphism. All of these are examples of circle endomorphisms having bounded geometry. 
In Section~\ref{srcehbg}, we study the symmetric rigidity for circle endomorphisms having bounded geometry and prove our main theorem (Theorem~\ref{main}). 
Finally, in the same section we state several corollaries (Corollary~\ref{cor1}, Corollary~\ref{cor2}, Corollary~\ref{cor3}, and Corollary~\ref{cor4}) of our main theorem.

\bigskip
\noindent {\bf Acknowledgment:} We would like to thank Professor Frederick Gardiner for help and communications during this research. 

\bigskip

\section{Circle Endomorphisms Having Bounded Geometry}~\label{cehbg}

Let $T=\{ z\in {\mathbb C}\; |\; |z|=1\}$
be the unit circle in the complex plane ${\mathbb C}$.
Let $m$ be the Lebesgue probability measure on $T$ (i.e. a Haar measure on $T$).
Suppose
$$
  f: T\to T
$$
is an orientation-preserving covering map of degree $d\geq 2$. We
call it a circle endomorphism. Suppose
$$
  h: T\to T
$$
is an orientation-preserving homeomorphism. We call it a circle homeomorphism.
Every circle endomorphism $f$ has at least one fixed point.
By conjugating $f$ by a rotation of the circle if necessary,
we assume that $1$ is a fixed point of $f$, that is, $f(1)=1$.

The universal cover of $T$ is the real line ${\mathbb R}$ with a
covering map
$$
  \pi (x) = e^{2\pi i x}: {\mathbb R} \to T.
$$
In this way, we can think the unit interval $[0,1]$ as the unit circle $T$. 

Then every circle endomorphism $f$ can be lifted to a homeomorphism
$$
  F: {\mathbb R}\to {\mathbb R}, \quad F(x+1)=F(x)+d,
  \quad \forall x\in {\mathbb R}.
$$
\begin{equation*}
  \begin{tikzcd}
    \mathbb{R}\arrow[r, "F"]\arrow[d, "\pi"] & \mathbb{R}\arrow[d, "\pi"] \\
    T\arrow[r, "f"]& T
    \end{tikzcd}
\end{equation*}
We will assume that $F(0)=0$ so that there is
a one-to-one correspondence between $f$ and $F$. 
Therefore, we also call such a map $F$ a circle endomorphism.

Similarly, every circle homeomorphism $h$ can be lifted to an
orientation-preserving homeomorphism
$$
  H: {\mathbb R}\to {\mathbb R}, \quad H(x+1)=H(x)+1,
  \quad \forall x\in {\mathbb R}.
$$
\begin{equation*}
  \begin{tikzcd}
    \mathbb{R}\arrow[r, "H"]\arrow[d, "\pi"] & \mathbb{R}\arrow[d, "\pi"] \\
    T\arrow[r, "h"]& T
    \end{tikzcd}
\end{equation*}
We will assume that $0\leq H(0)<1$ so that there is a
one-to-one correspondence between $h$ and $H$.
Therefore, we also call such a map $H$ a circle homeomorphism.
Since we only consider circle homeomorphisms as conjugacies of circle endomorphisms in this paper, we assume $h(1)=1$ (equivalently,  $H(0)=0$). We use $id$ and $ID$ to denote the identity circle homeomorphism and its lift to $\mathbb{R}$, respectively. That is, $id(z)=z$ and $ID(x)=x$.

Let $\CE(d)$ be the space of all circle endomorphisms $f$ (or $F$) of degree $d\geq 2$ fixing $1$ (or $0$) and let $\CH$ be the space of all circle homeomorphisms $h$ (or $H$) fixing $1$ (or $0$). 

%The Lebesgue length of an interval $I\in\mathbb{R}$ is denoted by $|I|$. It gives the normalized arc-length on the unit circle $T$ and  a probability measure on $T$ (or $[0,1]$) which we denote it as $m$.  

\medskip
\begin{definition}~\label{def1}
  A circle homeomorphism $h\in \CH$ is called quasisymmetric (refer to~\cite{Ah})
  if there exists a constant $M\geq1$ such that
  \begin{equation*}
    \frac{1}{M}\leq\frac{H(x+t)-H(x)}{H(x)-H(x-t)}\leq M
    \quad\forall x\in\mathbb{R},\;\;\forall t>0.
  \end{equation*}
  It is called symmetric (refer to~\cite{GS})
  if there exists a positive bounded function $\epsilon(t)$
  such that $\epsilon(t)\to0$ as $t\to0^{+}$ and
  \begin{equation*}
    \frac{1}{1+\epsilon(t)}\leq
    \frac{H(x+t)-H(x)}{H(x)-H(x-t)}\leq1+\epsilon(t)
    \quad\forall x\in\mathbb{R},\;\; \forall t>0.
  \end{equation*}
\end{definition}

\medskip
\begin{definition}~\label{def2}
  A circle endomorphism $f\in \CE(d)$ is called uniformly quasisymmetric (refer to~\cite{JiangNote,JiangGibbs})
  if there exists a constant $M\geq1$ such that
  \begin{equation*}
    \frac{1}{M}\leq\frac{F^{-n}(x+t)-F^{-n}(x)}{F^{-n}(x)-F^{-n}(x-t)}\leq M
    \quad\forall n\geq1,\;\;\forall x\in\mathbb{R},\;\;\forall t>0.
  \end{equation*}
  It is called uniformly symmetric (refer to~\cite{GJ,JiangGibbs}) if there exists a positive bounded function $\epsilon(t)$
  such that $\epsilon(t)\to0$ as $t\to0^{+}$ and  
   \begin{equation*}
    \frac{1}{1+\epsilon(t)}\leq\frac{F^{-n}(x+t)-F^{-n}(x)}{F^{-n}(x)-F^{-n}(x-t)}\leq 1+\epsilon(t)
    \quad\forall n\geq1,\;\;\forall x\in\mathbb{R},\;\;\forall t>0.
  \end{equation*}\end{definition}

An example of a symmetric (and quasisymmetric) circle homeomorphism is a $C^{1}$ circle diffeomorphism. However, in general, a symmetric (or quasisymmetric) circle homeomorphism may not be differentiable, and may even be totally singular with respect to the Lebesgue measure.  

If a circle endomorphism $f$ is differentiable and the derivative $F'$ is positive, then we can define the modulus of continuity,
$$
\omega (t) =\sup_{|\xi-\eta|\leq t} |\log F' (\xi)- \log F'(\eta)|.
$$
We say that $f$ is $C^{1+Dini}$ if $\omega(t)$ satisfies the Dini condition that 
$$
\int_{0}^{1} \frac{\omega (t)}{t} \;dt < \infty.
$$
We say that $f$ is $C^{1+\alpha}$ for some $0<\alpha\leq 1$ if the derivative $F'$ is an $\alpha$-H\"older continuous function. It is clear that a $C^{1+\alpha}$ map is a $C^{1+Dini}$ map.
 We say that $f$ is expanding if there are two constants $C>0$ and $\lambda >1$ such that 
 $$
 |(F^{n})' (z)| \geq C\lambda^{n},  \quad \forall z\in T, \;\; \forall n\geq 1.
 $$ 
 
 \medskip
 \begin{example}~\label{example1}
 Every $C^{1+Dini}$ expanding circle endomorphism of degree $d\geq 2$ is uniformly symmetric.
 \end{example}  
 
 See~\cite{JiangGibbs} for a proof of this example. However, in general, a uniformly symmetric (or quasisymmetric) 
 circle endomorphism may not be differentiable, may not be absolutely continuous,
 and may even be totally singular with respect to the Lebesgue measure.  

Let $\QS$ be the space of all quasisymmetric circle homeomorphisms in $\CH$. Let $\S$ be the space of all symmetric circle homeomorphisms in $\CH$. Then we have from Definition~\ref{def1}
$$
\S\subset \QS\subset \CH.
$$

Let $\UQCE(d)$ be the space of all uniformly quasisymmetric circle endomorphisms in $\CE(d)$. Let $\USCE(d)$ be the space of all uniformly symmetric circle endomorphisms in $\CE(d)$. Let $\CED (d)$ be the space of all $C^{1+Dini}$ expanding circle endomorphisms of degree $d\geq 2$. Then we have from Example~\ref{example1} and Definition~\ref{def2}
$$
\CED (d) \subset \USCE(d)\subset \UQCE(d)\subset \CE(d).
$$
 
\medskip
\begin{definition}~\label{lm}
We say $f\in \CE(d)$ preserves the Lebesgue measure $m$ if 
\begin{equation}~\label{inv}
m(f^{-1} (A)) =m(A)
\end{equation}
holds for all Borel subsets $A\subseteq T$.
\end{definition}
 
Henceforth, in order to avoid confusion, we will consistently use 
$$
[0,1]/\{0\sim 1\}= {\mathbb R} \pmod{1}
$$ 
to mean the unit circle. Likewise, we will consistently use
$$
f=F \pmod{1}: [0,1]/\{0\sim 1\}\to [0,1]/\{0\sim 1\}
$$ 
to mean a circle endomorphism and 
$$
h=H \pmod{1}: [0,1]/\{0\sim 1\}\to [0,1]/\{0\sim 1\}
$$ 
to mean a circle homeomorphism.

For any $f\in \CE(d)$, the preimage $f^{-1}(0)$ of the fixed point $0$
partitions $[0,1]$ into $d$ closed and ordered intervals $I_{0}$, $I_{1}$, $\cdots$, $I_{d-1}$  (see Figure 1).
Let
$$
  \eta_{1}=\{ I_{0}, I_{1}, \cdots, I_{d-1}\}.
$$
Then $\eta_1$ is a Markov partition. That is,
\begin{enumerate}[(i)]
  \item $[0,1]=\cup_{i=0}^{d-1} I_{i}$;
 \item $I_{i}$ and $I_{j}$ have pairwise disjoint interiors for any $0\leq i<j\leq d-1$;   
 \item $f(I_{i})=[0,1]$ for every $0\leq i\leq d-1$;
 \item the restriction of $f$ to the interior of $I_{i}$
    is injective for every $0\leq i\leq d-1$.
\end{enumerate}

\begin{figure}[ht]
    \includegraphics[width=3in]{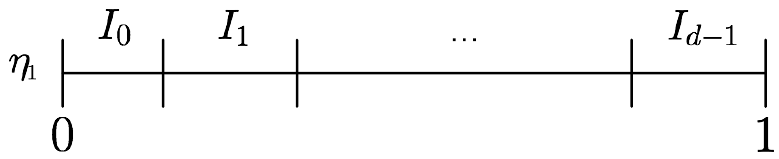}
  \caption{The initial Markov partition.}
\end{figure} 

The preimage $f^{-n}(0)$ of the fixed point $0$
partitions $[0,1]$ into $d^{n}$ closed intervals $I_{w_{n}}$ labeled by 
$$
w_{n}=i_{0}i_{1}\ldots i_{n-1}\in\Sigma_n =\prod_{k=0}^{n-1}\{0,1,\ldots,d-1\}
$$
and defined inductively as
$$
  f^{k}(I_{w_{n}}) \subset I_{i_{k}},\;\; \forall0\leq k\leq n-2, \quad \hbox{and}\quad f^{n-1}(I_{w_{n}}) =I_{i_{n-1}}.
$$
Let
$$
\eta_{n} =\{ I_{w_{n}} \;|\; w_{n}=i_{0}i_{1}\ldots i_{n-1}\in\Sigma_n\}.
$$
Then $\eta_n$ is also a Markov partition. That is,
\begin{enumerate}
  \item $[0,1]=\cup_{w_{n}\in \Sigma_{n}} I_{w_{n}}$;
  \item intervals in $\eta_{n}$ have pairwise disjoint interiors; 
  \item $f^{n}(I_{w_{n}})=[0,1]$ for every $w_{n}\in \Sigma_{n}$;
  \item the restriction of $f^{n}$ to the interior of $I_{w_{n}}$
    is injective for every $w_{n}\in \Sigma_{n}$.
\end{enumerate}
   
\medskip
\begin{remark}
Suppose ${\mathcal A}$ and ${\mathcal B}$ are two partitions of $T$. The partition 
$$
A\vee B =\{ A\cap B\;|\; A\in {\mathcal A}, B\in {\mathcal B}\}
$$
is the finer partition from ${\mathcal A}$ and ${\mathcal B}$. Then we have that
$$
\eta_{n} =\vee_{k=1}^{n} f^{-k} \eta_{1}.
$$
\end{remark}   
   
Let $\s$ be the left-shift map and let $\s^*$ be the right-shift map on $\Sigma_n$, that is,
$$
  \s(\o_n)=\s(i_0i_1\ldots i_{n-2}i_{n-1})=i_1\ldots i_{n-2}i_{n-1}
$$
and
$$
  \s^{*}(\o_n)=\s^*(i_0i_1\ldots i_{n-2}i_{n-1})=i_0i_1\ldots i_{n-2}.
$$
Here we assume $w_{0}=\emptyset$ and $I_{w_{0}}=[0,1]$ and $\s (w_{1}) =w_{0}$ and $\s^{*} (w_{1}) =w_{0}$. 
Then we have 
$$
  I_{w_n}=\cup_{k=0}^{d-1}I_{w_nk}=\cup_{w_{n+1}\in (\s^{*})^{-1} (w_{n})} I_{w_{n+1}} 
$$
and
$$
f^{-1}(I_{w_n})=\cup_{k=0}^{d-1}I_{k\o_n}=\cup_{w_{n+1}\in \s^{-1} (w_{n})} I_{w_{n+1}}.
$$

\begin{figure}[ht]
    \includegraphics[width=3.5in]{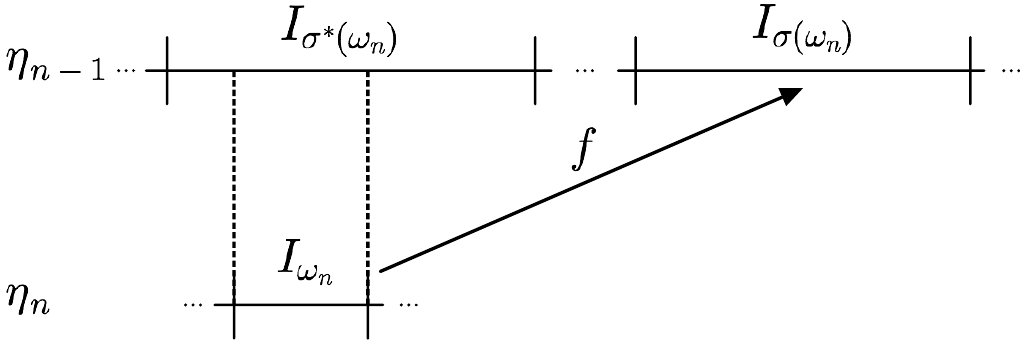}
  \caption{$I_{w_{n}}\subset I_{\s^{*}(w_{n})}$ and $f(I_{w_{n}})=I_{\s (w_{n})}$.}   
\end{figure} 

\medskip
\begin{definition}~\label{bg}
  A circle endomorphism $f$ is said to have bounded geometry (refer to~\cite{JiangNote,JiangGibbs}) %as well as~\cite{AHJW})
 if there is a constant $C>1$ such that 
  \begin{equation}\label{bgdef}
    \frac{|I_{\s^*(\o_n)}|}{|I_{\o_n}|}\leq C,
    \quad\forall\o_n\in\Sigma_n,\;\;\forall n\geq1. \quad \hbox{(See Figure 2)}
  \end{equation}
\end{definition}

\medskip
Let $\BGCE(d)$ be the space of all $f\in \CE(d)$ having bounded geometry. Then we have (refer to~\cite{JiangNote,JiangGibbs})
$$
\CED (d) \subset \USCE(d) \subset \UQCE(d)  \subset \BGCE(d)\subset \CE(d).
$$
We know that $\USCE(d)$ is not equal to $\UQCE(d)$ (refer to~\cite{JiangNote}). Also, $\UQCE(d)$ is not equal to $\BGCE(d)$. %(see~\cite{JiangNote,HJW,AHJW}).  
For example, for any $\alpha\in(1/2,1)$, the piecewise-linear degree 2 circle endomorphism%\ref{fig:CircleEndoAlpha})
\begin{equation*}
  f_\alpha(x)=\left\{ \begin{array}{ll}
      f_\alpha(x+1)-2&\mbox{if }x<0\\
      x/\alpha&\mbox{if }0\leq x<\alpha\\
      1+(x-\alpha)/(1-\alpha)&\mbox{if }\alpha\leq x<1\\
      f_\alpha(x-1)+2&\mbox{if }1\leq x
  \end{array}\right.
\end{equation*}
has bounded geometry because
\begin{equation*}
  \frac{|I_{\s^*(\o_n)}|}{|I_{\o_n}|}\in\{1/\alpha, 1/(1-\alpha)\},
    \quad\forall\o_n\in\Sigma_n,\forall n\geq1.
\end{equation*}
However, $f_\alpha$ is not uniformly quasisymmetric since for any $0<t<1-\alpha$,
\begin{equation*}
  \frac{f_\alpha^{-n}(0+t)-f_\alpha^{-n}(0)}{f_\alpha^{-n}(0)-f_\alpha^{-n}(0-t)}=\left(\frac{\alpha}{1-\alpha}\right)^n\to\infty\mbox{ as }n\to\infty.
\end{equation*}
\begin{remark}
The property of uniform quasisymmetry for a circle endomorphism can be equivalently characterized in terms of its sequence of nested partitions $\{\eta_n\}$ alone by saying that the circle enedomorphism has bounded nearby geometry. The precise definition of bounded nearby geometry is given and its equivalence to uniform quasisymmetry is proved in~\cite{JiangGeo,JiangBook}. For more on circle endomorphisms with bounded geometry and/or bounded nearby geometry, see also~\cite{JiangNote,JiangGibbs,HuThesis,AdamskiThesis}).%(see~\cite{AHJW}).
\end{remark}

For $f\in \BGCE(d)$, let
$$
\tau_{n} =\max\{ |I_{w_{n}}| \;|\; w_{n}\in \Sigma_{n}\}.
$$
Then from Definition~\ref{bg}, we have a constant $0<\tau<1$ such that 
\begin{equation}~\label{expdecay}
\tau_{n} \leq \tau^{n}, \quad \forall n\geq 1.
\end{equation}
It follows that any two maps $f, g\in \BGCE(d)$ are topologically conjugate. That is, there is an $h\in \CH$ such that 
\begin{equation}~\label{ceg}
f \circ h = h\circ g.
\end{equation}
Here $h$ is called the conjugacy from $f$ to $g$, and when $h\in \S$ we call it a symmetric conjugacy.
In the special case that both $f$ and $g$ are in $\UQCE(d)$, we further know that the conjugacy $h\in \QS$. However, as long as at least one of the maps is not in $\UQCE(d)$, 
the conjugacy $h$ may not be quasisymmetric. Refer to~\cite{JiangNote,JiangGibbs,HJW,HuThesis,AdamskiThesis}.

\section{Symmetric Rigidity, the Proof of the Main Result} ~\label{srcehbg}

We start %the proof of Theorem~\ref{main} 
with the following lemma.

\medskip
\begin{lemma}~\label{tail}
 Suppose $f\in \BGCE(d)$. Let $\{\eta_{n}\}_{n=1}^{\infty}$ be the corresponding sequence of partitions.  
 Suppose $I_{w_{n}}\in \eta_{n}$ is a fixed partition interval for some $n\geq 1$. Then 
 $$
\lim_{k\to \infty} \sum_{w_{n}^{1} \not=w_{n}} \cdots \sum_{w_{n}^{k}\not=w_{n}} |I_{w_{n}^{1}\cdots w_{n}^{k}}|=0,
$$ 
where the $\omega_n^i$ are all words of length $n$.
 \end{lemma}
 
 \begin{proof}
 From the definition of bounded geometry (Definition~\ref{bg}), we have that
 $$
 |I_{w_{n}}| \geq A=\frac{1}{C^{n}}.
 $$
 Since 
 $$
 \bigcup_{w_{n}^{1}\not=w_{n}} I_{w_{n}^{1}}=\overline{[0,1]\setminus I_{w_{n}}},
 $$  
 we get
 $$
 \sum_{w_{n}^{1}\not=w_{n}} |I_{w_{n}^{1}}| = 1-|I_{w_{n}}|  \leq 1-A.
 $$
 
For any $w_{n}^{1}$, we have that $I_{w_{n}^{1}w_{n}} \subset I_{w_{n}^{1}}$. Because of bounded geometry, we further have 
$$
|I_{w_{n}^{1}w_{n}}| \geq A |I_{w_{n}^{1}}|.
$$ 
Since 
$$
\bigcup_{w_{n}^{2}\not=w_{n}} I_{w_{n}^{1}w_{n}^{2}}=\overline{I_{w_{n}^{1}}\setminus I_{w_{n}^{1}w_{n}}},
$$
we have
$$
\sum_{w_{n}^{2}\not=w_{n}} |I_{w_{n}^{1}w_{n}^{2}}|= |I_{w_{n}^{1}}|-|I_{w_{n}^{1}w_{n}}| \leq |I_{\omega_n^1}| - A|I_{\omega_n^1}| = (1-A) |I_{w_{n}^{1}}|.
$$
This implies that 
$$
\sum_{w_{n}^{1}\not=w_{n}} \sum_{w_{n}^{2}\not= w_{n}} |I_{w_{n}^{1}w_{n}^{2}}|\leq (1-A)\sum_{w_{n}^{1}\not=w_{n}} |I_{w_{n}^{1}}|\leq (1-A)^{2}.
$$ 

Inductively, suppose we know
$$
\sum_{w_{n}^{1}\not=w_{n}} \cdots \sum_{w_{n}^{k-1}\not=w_{n}} |I_{w_{n}^{1}\ldots w_{n}^{k-1}}| 
\leq (1-A)^{k-1}
$$
for $k\geq 3$.
Then 
$$
\sum_{w_{n}^{1}\not=w_{n}} \cdots \sum_{w_{n}^{k}\not=w_{n}} |I_{w_{n}^{1}\ldots w_{n}^{k}}|
 =\sum_{w_{n}^{1}\not=w_{n}} \cdots \sum_{w_{n}^{k-1}\not=w_{n}} \Big( |I_{w_{n}^{1}\cdots w_{n}^{k-1}}| - |I_{w_{n}^{1}\cdots w_{n}^{k-1}w_{n}}|\Big)
$$
$$
= \sum_{w_{n}^{1}\not=w_{n}} \cdots \sum_{w_{n}^{k-1}\not=w_{n}} |I_{w_{n}^{1}\cdots w_{n}^{k-1}}| \Big( 1-\frac{  |I_{w_{n}^{1}\cdots w_{n}^{k-1}w_{n}}|}{  |I_{w_{n}^{1}\cdots w_{n}^{k-1}}|}\Big).
$$
Notice that $I_{w_{n}^{1}\cdots w_{n}^{k-1}w_{n}} \subset I_{w_{n}^{1}\cdots w_{n}^{k-1}}$. The definition of bounded geometry implies that
$$
\frac{ |I_{w_{n}^{1}\cdots w_{n}^{k-1}w_{n}}|}{ |I_{w_{n}^{1}\cdots w_{n}^{k-1}}|}\geq A.
$$
This implies that 
$$
1-\frac{ |I_{w_{n}^{1}\cdots w_{n}^{k-1}w_{n}}|}{ |I_{w_{n}^{1}\cdots w_{n}^{k-1}}|}\leq 1-A.
$$
Thus
$$
\sum_{w_{n}^{1}\not=w_{n}} \cdots \sum_{w_{n}^{k}\not=w_{n}} |I_{w_{n}^{1}\ldots w_{n}^{k}}| \leq (1-A) \sum_{w_{n}^{1}\not=w_{n}} \cdots \sum_{w_{n}^{k-1}\not=w_{n}} |I_{w_{n}^{1}\cdots w_{n}^{k-1}}|\leq (1-A)^{k}.
$$
Letting $k\to\infty$, this proves the lemma.
 \end{proof}

Given a partition interval $I_{w_{n}}\in \eta_{n}$ for some $n\geq 1$, define 
$$
C(I_{w_{n}}) = \{ x\in [0,1]\;|\; f^{kn}(x) \not\in I_{w_{n}}, k=0, 1, 2, \cdots\}=\bigcap_{i=1}^\infty \left( \bigcup_{\substack{  \omega_n^j \neq \omega_n \\ 1 \leq j \leq i}}I_{\omega_n^1\ldots\omega_n^i}\right).
$$
A consequence of Lemma~\ref{tail} is the following. 

\medskip
\begin{corollary}~\label{zero}
 Suppose $f\in \BGCE(d)$. Then the set $C(I_{w_{n}})$ has zero Lebesgue measure. That is, $m(C(I_{w_{n}}))=0$.
\end{corollary}

Suppose $f$ and $g$ are both in $\BGCE(d)$ and $h\in \CH$ is the conjugacy from $f$ to $g$. Define the number
$$
1\leq \Phi=\sup_{I\subseteq [0,1]} \frac{|h(I)|}{|I|}\leq \infty
$$
and the set
\begin{equation}~\label{no1}
 X=\{ x\in [0,1]\;|\; \exists I_{k}^{x}=[a_{k}, b_{k}], \lim_{k\to \infty} a_{k}=\lim_{k\to \infty} b_{k} =x, \lim_{k\to \infty} \frac{|h(I_{k}^x)|}{|I_{k}^x|} =\Phi\}
 \end{equation}
We would like to note that, in general, $\Phi=\infty$ and when $\Phi<\infty$, $h$ is a Lipschitz conjugacy.

 \medskip
 \begin{remark} 
 Similarly, we can also define  
$$
0\leq \phi= \inf_{I\subseteq [0,1]} \frac{|h(I)|}{|I|}\leq 1
$$ 
and use $\phi$ to prove Theorem~\ref{main}.
\end{remark}

\medskip
\begin{lemma}~\label{nonempty}
Suppose $f, g\in \BGCE(d)$. Then $X$ is a non-empty subset of $T$.
\end{lemma}

\begin{proof}
Suppose $\{I_{k}=[a_{k}, b_{k}]\}_{k=1}^{\infty} $ is a sequence of intervals such that 
$$
\lim_{k\to \infty} \frac{|h(I_{k})|}{|I_{k}|}=\Phi.
$$
By taking a subsequence if necessary, we assume that $\{ a_{k}\}_{k=1}^{\infty}$ and  $\{b_{k}\}_{k=1}^{\infty}$ are two convergent sequences of numbers and $a=\lim_{k\to \infty} a_{k}$ and $b=\lim_{k\to \infty} b_{k}$. 

If $a=b=x$, then $x\in X$ and $X\not=\emptyset$. Note that if $\Phi=\infty$ then $a=b$.

If $a<b$, then $I=[a, b]$ is a non-trivial interval such that 
\begin{equation}~\label{reach1}
\frac{|h(I)|}{|I|}=\Phi. 
\end{equation}
In this case, we claim that for any non-trivial subinterval $I'\subset I$, $|h(I')|/|I'|=\Phi$. 
The claim implies that $I\subset X$, and thus, $X\not=\emptyset$. 
Now we prove the claim as follows. Let $I'=[a',b']$ with $a\leq a'<b'\leq b$. Let $L=[a,a']$ and $R=[b',b]$. Then we have $I=L\cup I' \cup R$ and $h(I) =h(L)\cup h(I')\cup h(R)$. 
Assume $|h(I')|/|I'|<\Phi$. Then, since $|h(L)| \leq \Phi|L|$, and $|h(R)| \leq \Phi|R|$, we have
$$
\frac{|h(I)|}{|I|}= \frac{|h(L)|+|h(I')| +|h(R)|}{|L|+|I'|+|R|} <\Phi.
$$ 
This is a contradiction. Thus we have proved the claim and completed the proof.
\end{proof}

Furthermore, under the assumption that both $f$ and $g$ preserve the Lebesgue measure $m$, we have the following stronger result. 

\medskip
\begin{lemma}~\label{dense}
Suppose $f, g\in \BGCE(d)$ both preserve the Lebesgue measure $m$. Then $X$ is dense in $[0,1]$. That is,  $\overline{X}=[0,1]$.
\end{lemma}

\begin{proof}   
%Since $X$ is not empty, we have an $x\in X$. Let $I_{m}^{x}=[a_{m}^{x}, b_{m}^{x}]$ be a sequence of intervals such that $\lim_{m\to \infty} a_{m}^{x}=\lim_{m\to \infty} b_{m}^{x} =x$ and such that $\lim_{m\to \infty} \frac{|h(I_{m}^{x})|}{|I_{m}^{x}|} =\Phi$.
We will prove that for any $n\geq 1$ and for any partition interval $I_{w_{n}}\in\eta_{n}$, $I_{w_{n}}\cap X\not=\emptyset$. 
It will then follow from inequality (\ref{expdecay}) that $\overline{X}=[0,1]$.
We prove it by contradiction.

Assume we have a partition interval $I_{w_{n}}$ such that $I_{w_{n}}\cap X=\emptyset$. Then we can find a number $D<\Phi$ such that 
\begin{equation}~\label{d}
\frac{|h(I)|}{|I|} \leq D
\end{equation}
for all $I\subset I_{w_{n}}$. 

Since $X\not=\emptyset$, we have an interval $I^{D}\subseteq [0,1]$
such that 
\begin{equation}~\label{big}
\frac{|h(I^{D})|}{|I^{D}|} > D.
\end{equation}
We pull back $I^{D}$ by $f^{n}$ to get
$f^{-n}(I^{D})=\cup_{w_{n}^{1}} I_{w_{n}^{1}}^{D}$, where $I^{D}_{w_{n}^{1}}\subset I_{w_{n}^{1}}\in \eta_{n}$ and $f^{n} (I^{D}_{w_{n}^{1}}) =I^{D}$.

Since both $f$ and $g$ preserve the Lebesgue measure $m$, for all $k\geq2$ we have
\begin{align}~\label{sum1}
|I^{D}| &= |I_{\omega_n}^D| + \sum_{\omega_n^1 \neq \omega_n}|I_{\omega_n^1}^D| = |I_{\omega_n}^D| + \sum_{\omega_n^1 \neq \omega_n}|f^{-n}(I_{\omega_n^1}^D)| \nonumber \\
        &= |I_{\omega_n}^D| + \sum_{\omega_n^1 \neq \omega_n}\left( |I_{\omega_n \omega_n^1}^D| + \sum_{\omega_n^2 \neq \omega_n}|I_{\omega_n^2\omega_n^1}^D| \right) \nonumber \\
        &= |I_{\omega_n}^D| + \sum_{\omega_n^1 \neq \omega_n}|I_{\omega_n \omega_n^1}^D| + \sum_{\omega_n^2\neq\omega_n}\sum_{\omega_n^1\neq\omega_n}|I_{\omega_n^2 \omega_n^1}^D| = \ldots \nonumber \\
        &= |I_{w_{n}}^{D}| +  \sum_{l=1}^{k-1}  \sum_{w_{n}^{l}\not=w_{n}} \cdots \sum_{w_{n}^{1}\not= w_{n}}  |I_{w_{n}w_{n}^{l}\ldots w_{n}^{1}}^{D}| + \sum_{w_{n}^{k}\not=w_{n}} \cdots \sum_{w_{n}^{1}\not= w_{n}} |I^{D}_{w_{n}^{k}\ldots w_{n}^{1}}|
\end{align}
and, similarly, 
\begin{equation}~\label{sum2}
|h(I^{D})|=|h(I_{w_{n}}^{D})| + \sum_{l=1}^{k-1}  \sum_{w_{n}^{l}\not=w_{n}} \cdots \sum_{w_{n}^{1}\not= w_{n}}  |h(I_{w_{n}w_{n}^{l}\ldots w_{n}^{1}}^{D})| +  \sum_{w_{n}^{k}\not=w_{n}} \cdots \sum_{w_{n}^{1}\not= w_{n}} |h(I^{D}_{w_{n}^{k}\ldots w_{n}^{1}})|.
\end{equation}
\begin{figure}[ht]
\begin{center}\includegraphics[width=5in]{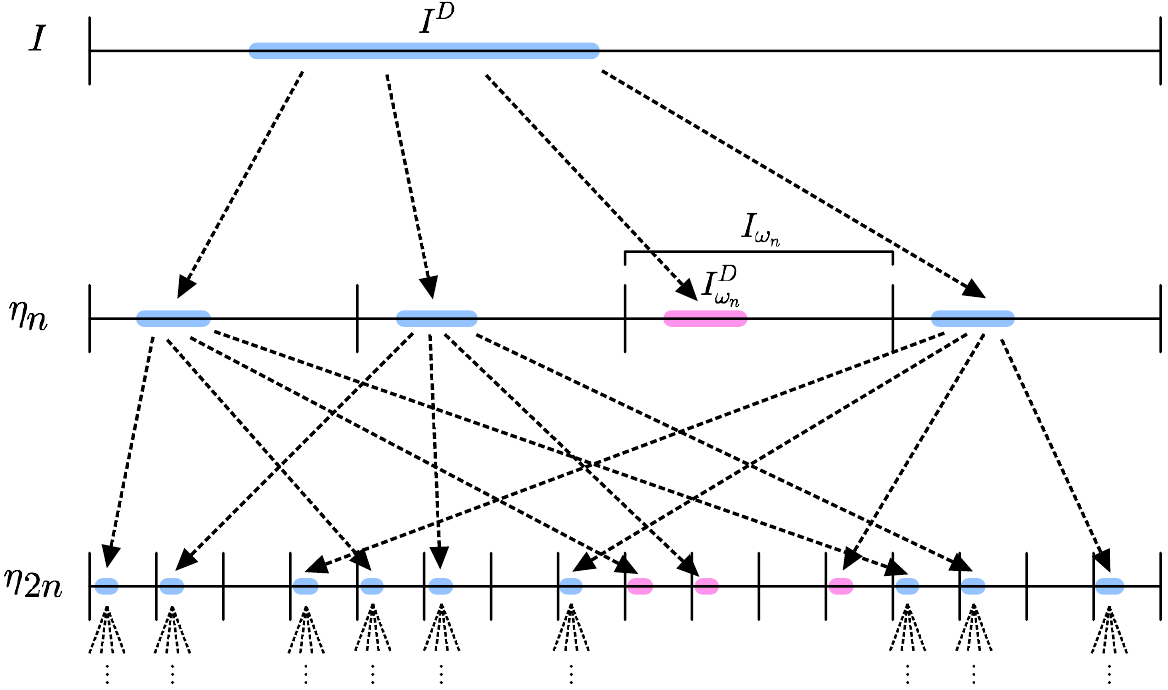}
\caption{The interval $I^D$ has a preimage under $f^n$ composed of $d^n$ intervals, one of which is a subset of $I_{\omega_n}$. Similarly, each of these preimage-intervals that is not a subset of $I_{\omega_n}$ has a preimage under $f^n$ composed of $d^n$ intervals, one of which is a subset of $I_{\omega_n}$. Equation~(\ref{sum1}) says that the length of $I^D$ is equal to the sum of the lengths of all blue intervals belonging to the same arbitrary level plus the lengths all pink intervals belonging to that same level or any previous level.}~\label{preimages}
\end{center}\end{figure}
See Figure~\ref{preimages}. Because $I_{w_{n}}^{D}$ and $I_{w_{n}w_{n}^{l}\ldots w_{n}^{1}}^{D}$ are sub-intervals of $I_{w_{n}}$, (\ref{d}) says that 
$$
\frac{|h(I^{D}_{w_{n}})| }{|I^{D}_{w_{n}}|},  \;\;
\frac{|h(I^{D}_{w_{n}w_{n}^{l}\cdots w_{n}^{1}})| }{|I^{D}_{w_{n}w_{n}^{l}\cdots w_{n}^{1}}|} \leq D\;\; \forall \; l\geq 1.
$$
This implies that 
\begin{equation}~\label{inq}
\frac{|h(I_{w_{n}}^{D})| + \sum_{l=1}^{k-1}  \sum_{w_{n}^{l}\not=w_{n}} \cdots \sum_{w_{n}^{1}\not= w_{n}}  |h(I_{w_{n}w_{n}^{l}\ldots w_{n}^{1}}^{D})|}{|I_{w_{n}}^{D}| +  \sum_{l=1}^{k-1}  \sum_{w_{n}^{l}\not=w_{n}} \cdots \sum_{w_{n}^{1}\not= w_{n}}  |I_{w_{n}w_{n}^{l}\ldots w_{n}^{1}}^{D}|}\leq D\;\;\forall\; k\geq 2.
\end{equation}
From Lemma~\ref{tail}
\begin{equation}~\label{tail1}
\lim_{k\to \infty} \sum_{w_{n}^{k}\not=w_{n}} \cdots \sum_{w_{n}^{1}\not= w_{n}} |I^{D}_{w_{n}^{k}\cdots w_{n}^{1}}|=0
\end{equation}
and 
\begin{equation}~\label{tail2}
\lim_{k\to \infty}  \sum_{w_{n}^{k}\not=w_{n}} \cdots \sum_{w_{n}^{1}\not= w_{n}} |h(I^{D}_{w_{n}^{k}\cdots w_{n}^{1}})|=0.
\end{equation}
Now (\ref{sum1}), (\ref{sum2}), (\ref{inq}), (\ref{tail1}), and (\ref{tail2}) imply that 
$$
\frac{|h(I^{D})|}{|I^{D}|} 
% =\lim_{k\to \infty} \frac{\sum_{l=1}^{k} |I^{D}_{\underbrace{w_{n}\cdots w_{n}}_{l}}|+ \sum_{w_{n}^{k}\not=w_{n}} \cdots \sum_{w_{n}^{1}\not= w_{n}} |h(I^{D}_{w_{n}^{k}\cdots w_{n}^{1}})|}{\sum_{l=1}^{k} |I^{D}_{\underbrace{w_{n}\cdots w_{n}}_{l}}|+\sum_{w_{n}^{k}\not=w_{n}} \cdots \sum_{w_{n}^{1}\not= w_{n}} |h(I^{D}_{w_{n}^{k}\cdots w_{n}^{1}})|}
\leq D.
$$
This contradicts (\ref{big}). Thus our assumption that there exists a partition interval $I_{\omega_n}$ such that $I_{\omega_n}\cap X = \emptyset$ is false, and this proves the lemma.
\end{proof}

\begin{proof}[Proof of Theorem~\ref{main}]
We will prove that $\Phi=1$. Equivalently, we will prove that $\Phi>1$ cannot happen, regardless of $\Phi<\infty$ or $\Phi=\infty$.  

%In the view of Theorem~\ref{submain}, if (\ref{reach1}) holds, then $\Phi=1$ and $h=id$. Thus the rest of the proof will assume that (\ref{reach1}) does not hold.  

We proceed with a proof by contradiction.
Assume $\Phi>1$ (possibly $\infty$). Then we have two numbers $1< D_{1}< D_{2}<\Phi$.
Since $h$ is symmetric (see definition~\ref{def1}), there exists a positive bounded function $\epsilon(t)$
  such that $\epsilon(t)\to0$ as $t\to0^{+}$ and
  \begin{equation*}
    \frac{1}{1+\epsilon(t)}\leq
    \frac{|h(I)|}{|h(I')|}\leq1+\epsilon(t)
  \end{equation*}
holds for all closed intervals $I$ and $I'$ that have the same length $t>0$ and are adjacent, i.e. the right endpoint of one interval is the left endpoint of the other.
Fix $t_0$ such that 
\begin{equation}\label{t0}
  \epsilon(t)<\frac{D_2}{D_1}-1 \quad \forall t<t_0.
\end{equation}
Since $\overline{X}=[0,1]$ (Lemma~\ref{dense}), there exists an interval 
$I=[a, b]\subset (0,1)$ with $|I|=b-a<t_0$ such that 
\begin{equation}~\label{inq1}
\frac{|h(I)|}{|I|} > D_{2}.
\end{equation}
Let $L=[2a-b, a]\subset (0,1)$ and $R=[b, 2b-a]\subset (0, 1)$ (see figure \ref{fig3}). Then the intervals $L$ and $R$ 
are adjacent to $I$ and have the same length as $|I|$. 
It follows from (\ref{t0}) that
\begin{equation*}
  \frac{|h(R)|}{|R|} = \frac{|h(R)|}{|h(I)|} \cdot \frac{|h(I)|}{|I|} \cdot \frac{|I|}{|R|} > \frac{1}{1+\epsilon({b-a})} \cdot D_2 \cdot 1 > D_1
\end{equation*}
and
\begin{equation*}
  \frac{|h(L)|}{|L|} = \frac{|h(L)|}{|h(I)|} \cdot \frac{|h(I)|}{|I|} \cdot \frac{|I|}{|L|} > \frac{1}{1+\epsilon({b-a})} \cdot D_2 \cdot 1 > D_1.
\end{equation*}

%\begin{figure}[ht]
%    \includegraphics[width=5in]{fig3.pdf}
%  \vspace*{-3in}
%  \caption{$h$ is symmetric.}~\label{fig3}
%\end{figure} 

\begin{figure}[ht!]
\begin{center}\includegraphics[width=5in]{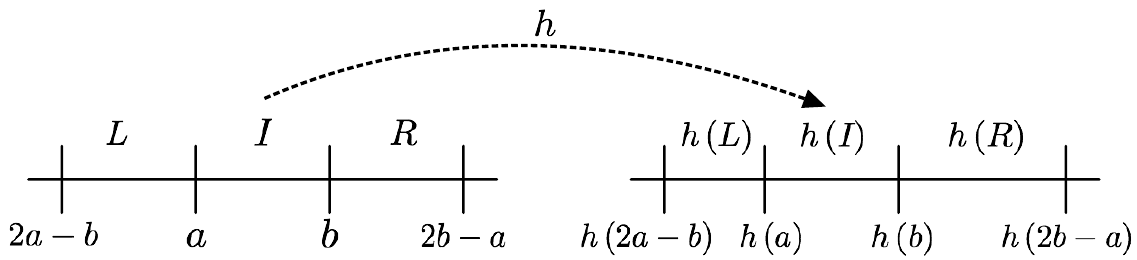}
\caption{$h$ is symmetric.}\label{fig3}
\end{center}\end{figure}

Now we want to show that 
\begin{equation*}
  \frac{|h([a,1])|}{|[a,1]|} > D_1.
\end{equation*}
Consider any interval $J=[b, c]\supset R$ with $2b-a\leq c\leq 1$ satisfying  
\begin{equation}~\label{inqj} 
 \frac{|h(J)|}{|J|} >D_1.
\end{equation}
If $c=1$, we have 
$$
\frac{|h([a, 1])|}{|[a,1]|} = \frac{|h(I\cup J)|}{|I\cup J|} >D_{1}.
$$
Then we have nothing further to prove.

If $c<1$, we have a number $\delta >0$ such that $c+\delta< 1$ and such that for any $x\in [c, c+\delta]$ we have 
\begin{equation*}%~\label{inq22} 
 \frac{|h([b,x])|}{|[b,x]|} >D_1.
\end{equation*}
Since $\overline{X} =[0,1]$ (Lemma~\ref{dense}), there is an interval $I_{1}=[a_{1}, b_{1}]\subset [c, c+\delta]$ with $|I_1|<t_0$ such that 
\begin{equation*}%~\label{inq22} 
 \frac{|h(I_{1})|}{|I_{1}|} >D_2.
\end{equation*}
Let $J_{1}=[b,a_{1}]$. Then we have three consecutive intervals $I$, $J_{1}$, and $I_{1}$ such that 
$$
\frac{|h([a, b_{1}])|}{|[a,b_{1}]|}  =\frac{|h(I \cup J_{1}\cup I_{1})|}{|I\cup J_{1}\cup I_{1}|}> D_{1}.
$$
(See Figure~\ref{newfig1}.)

 %\begin{figure}[ht]
 %   \includegraphics[width=6in]{newfig1.pdf}
 % \vspace*{-3.5in}
 % \caption{Construction of $J_{1}$ and $I_{1}$.}~\label{newfig1}
%\end{figure} 

\begin{figure}[ht!]
\begin{center}\includegraphics[width=\textwidth]{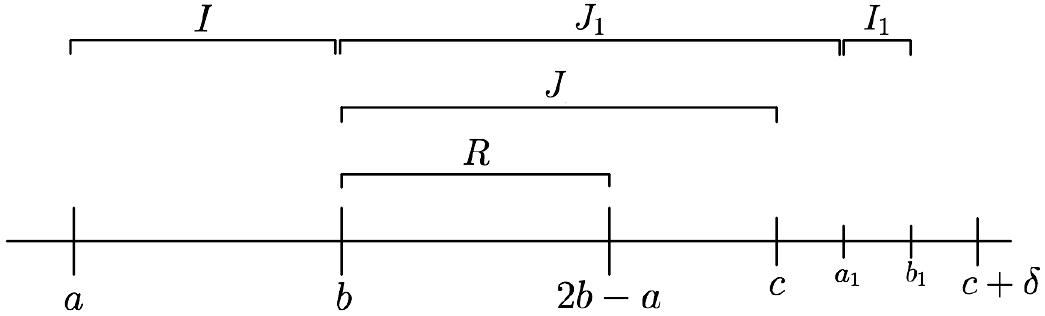}
\caption{Construction of $J_{1}$ and $I_{1}$.}~\label{newfig1}
\end{center}\end{figure}

Consider $I_{1}$ as a new $I$ and repeat the above construction. We get three consecutive intervals 
$I_{1}=[a_{1}, b_{1}]$, $J_{2}=[b_{1}, a_{2}]$, and $I_{2}=[a_{2}, b_{2}]$ such that
 $$
\frac{|h([a_{1}, b_{2}])|}{|[a_{1},b_{2}]|}  =\frac{|h(I_{1} \cup J_{2}\cup I_{2})|}{|I_{1}\cup J_{2}\cup I_{2}|}> D_{1}.
$$
Inductively, for every integer $n\geq 2$, we have three consecutive intervals 
$I_{n-1}=[a_{n-1}, b_{n-1}]$, $J_{n}=[b_{n-1}, a_{n}]$, and $I_{n}=[a_{n}, b_{n}]$ such that 
$$
\frac{|h([a_{n-1}, b_{n}])|}{|[a_{n-1},b_{n}]|}  =\frac{|h(I_{n-1} \cup J_{n}\cup I_{n})|}{|I_{n-1}\cup J_{n}\cup I_{n}|}> D_{1}.
$$
This implies that 
$$
\frac{|h([a, b_{n}])|}{|[a,b_{n}]|}  =\frac{|h(I\cup (\cup_{i=1}^{n} (J_{i}\cup I_{i})))|}{|I\cup (\cup_{i=1}^{n} (J_{i}\cup I_{i}))|}> D_{1}.
$$
If $b_{n}=1$, 
we have 
$$
\frac{|h([a, 1])|}{|[a,1]|} >D_{1}.
$$
Then we have nothing further to prove.

In the case that $b_{n}<1$ for all $n\geq 1$, since $\{b_{n}\}_{n=1}^{\infty}$ is a strictly increasing sequence in $[0,1)$, we have
$$
b_{\infty} =\lim_{n\to \infty} b_{n}\leq 1.
$$
and 
$$
\frac{|h([a, b_{\infty}])|}{|[a,b_{\infty}]|}  =\frac{|h(I\cup (\cup_{n=1}^{\infty} (J_{n}\cup I_{n})))|}{|I\cup (\cup_{n=1}^{\infty} (J_{n}\cup I_{n}))|}> D_{1}.
$$ 
%(See Figure~\ref{newfig1}.)
Since $b_{\infty}$ depends on the initially chosen interval $J$, we write it as $b_{\infty}(J)$.
Consider the set 
$$
{\mathcal B} =\{ b_{\infty} (J)\;|\; J \hbox{ satisfies (\ref{inqj})}\}
$$
Let $\beta =\sup {\mathcal B}$. We claim $\beta =1$. Otherwise, we take $J=[b, \beta]$. It satisfies (\ref{inqj}). Then $b_{\infty} (J) >\beta$. 
This contradiction proves the claim, and so 
$$
\frac{|h([a, 1])|}{|[a,1]|}  > D_{1}.
$$

Similarly, by using $L$ instead of $R$ and applying the procedure above, we get 
$$
\frac{|h([0, a])|}{|[0,a]|}  > D_{1}.
$$
Finally, we get the following contradiction.
$$
1= \frac{|h([0,1])|}{|[0,1]|} =\frac{|h([0, a)|+|h([a,1])|}{|[0, a]|+|[a,1]|} > D_{1}>1.
$$
The contradiction implies that $\Phi=1$. 

Since $\Phi=1$, we have that for any non-trivial interval $J\subset [0,1]$, $|h(J)|/|J| = 1$. Otherwise, if there is an interval $J\subset [0,1]$ such that $|h(J)|/|J| <1$, let $L\cup R=[0,1]\setminus J$. 
Then 
$$
1=\frac{|h([0,1])|}{|[0,1]|} = \frac{|h(L)|+|h(J)|+|h(R)|}{|L|+|J|+|R|} <1,
$$
since $|h(L)|\leq |L|$ and $|h(R)|\leq |R|$. This is a contradiction. Since $h(0)=0$, it follows that $h=id$. This  completes the proof of Theorem~\ref{main}. 
\end{proof}

%\section{Applications to Other Rigidity Results}~\label{app}
Theorem~\ref{main} has many consequences. In particular, we have affirmative answers to Conjecture~\ref{conj1} and Conjecture~\ref{conj2}, which we state as the following two corollaries.

\medskip
\begin{corollary}~\label{cor1}
Suppose $f, g\in \USCE(d)$ and both maps preserve the Lebesgue measure $m$. Suppose $h$ is the conjugacy from $f$ to $g$, and $h(1)=1$ . If $h\in \S$, then $h=id$.
\end{corollary}

\medskip
\begin{corollary}~\label{cor2}
Suppose $f, g\in \UQCE(d)$ and both maps preserve the Lebesgue measure $m$. Suppose $h$ is the conjugacy from $f$ to $g$, and $h(1)=1$.  If $h\in \S$, then $h=id$.
\end{corollary} 
 
Other consequences are new proofs of some known results in~\cite{JiangSymm} where we proved them by using transfer operators.  

\medskip
\begin{corollary}~\label{cor3}
Suppose $f$ and $g$ are $C^{1+Dini}$ expanding circle endomorphisms and both preserve the Lebesgue measure $m$. Suppose $h$ is the conjugacy from $f$ to $g$, and $h(1)=1$. If $h\in \S$, then $h=id$.
\end{corollary}
 
 \medskip
 \begin{corollary}~\label{cor4}
Suppose $f$ and $g$ are $C^{1+Dini}$ expanding circle endomorphisms and both preserve the Lebesgue measure $m$. Suppose $h$ is the conjugacy from $f$ to $g$, and $h(1)=1$. If $h$ is absolutely continuous, then $h=id$.
\end{corollary}
 
 \begin{proof}
 As shown in~\cite{SS} (see also~\cite{JiangSm,JiangRig}), if $h$ is absolutely continuous, then $h$ is a $C^{1}$ diffeomorphism. A $C^{1}$ diffeomorphism is symmetric. Now this corollary follows from Corollary~\ref{cor3}.
 \end{proof}

\end{document}